\documentclass[12pt]{article}

\usepackage{amsthm}
\usepackage{amssymb}
\usepackage{amsmath}
\usepackage{comment}
\usepackage{url} 
\usepackage{hyperref}
\usepackage[noabbrev,capitalise]{cleveref}

\usepackage{color}
\usepackage[normalem]{ulem}

\newcommand{\mc}[1]{\mathcal{#1}}

\renewcommand{\v}{\textup{\textsf{v}}}
\newcommand{\e}{\textup{\textsf{e}}}
\renewcommand{\d}{\textup{\textsf{d}}}

\theoremstyle{definition}

\theoremstyle{plain}
\newtheorem{thm}{Theorem}[section]
\newtheorem{lem}[thm]{Lemma}

\newtheorem{cor}[thm]{Corollary}

\newtheorem{conj}[thm]{Conjecture}

\title{A new upper bound on the chromatic number of graphs with no odd $K_t$ minor}

\author{
Sergey Norin\thanks{Department of Mathematics and Statistics, McGill University. E-mail: {\tt sergey.norin@mcgill.ca}. Supported by an NSERC grant.}\and \and
Zi-Xia Song\thanks{Department of Mathematics, University of Central Florida, Orlando, FL 32816, USA. E-mail: {\tt Zixia.Song@ucf.edu}. Supported by the National Science Foundation under Grant No. DMS-1854903.}}

\begin{document}

\maketitle

\begin{abstract}  Gerards and Seymour conjectured that every graph with no odd $K_t$ minor is $(t-1)$-colorable. This is a strengthening of the famous Hadwiger's Conjecture.
Geelen et al. proved that every graph with no odd $K_t$ minor is $O(t\sqrt{\log t})$-colorable. Using  the  methods the present authors and Postle recently developed for coloring graphs with no $K_t$ minor, we make the first improvement on  this bound by showing  that  every graph with no odd $K_t$ minor is $O(t(\log t)^{\beta})$-colorable for every $\beta > 1/4$.

\end{abstract}

\section{Introduction}

All graphs in this paper are finite, and have no loops or parallel edges.  Given graphs $G$ and $H$, we say that $G$ has \emph{an $H$ minor} if a graph isomorphic to $H$ can be obtained from a subgraph $G'$ of $G$ by contracting edges; and   $G$ has \emph{an odd $H$ minor} if, in addition, the set of contracted edges forms a cut in $G'$. (Note that the empty set $\emptyset$ is a cut.) We denote the complete graph on $t$ vertices by $K_t$. 

Gerards and Seymour in 1993  (see \cite[Section 6.5]{JenToft95}) made   the 
following conjecture.
\begin{conj}[Odd Hadwiger's Conjecture]
\label{odd} 
For every integer  $t \geq 1$, every graph 
with no odd  $K_{t}$ minor  is  $(t-1)$-colorable. 
\end{conj}

Conjecture~\ref{odd} substantially strengthens the famous  Hadwiger's Conjecture~\cite{Had43} which states that every graph with no $K_{t}$ minor is $(t-1)$-colorable. Conjecture \ref{odd} is trivially true  for  $t\le 3$. 
The case $t=4$  was proved by Catlin~\cite{Cat78}.   Guenin (see \cite{Sey16Survey}) announced the proof for
 the case $t=5$, but it has not yet been published.
Conjecture~\ref{odd} remains  open for  all $t \geq 6$. We refer the reader to a recent survey by Seymour~\cite{Sey16Survey} for further background.

The general upper bound  on the number of colors sufficient to properly color  graphs with no odd $K_{t}$ minor
was established by Geelen, Gerards, Reed, Seymour and Vetta~\cite{GGRSV08}, who proved the following. 

\begin{thm}[\cite{GGRSV08}]\label{t:GGRS} For every integer  $t \geq 1$,  every graph with no odd $K_t$ minor is $O(t \sqrt{\log t})$-colorable.
\end{thm}

Kawarabayashi~\cite{Kaw09} gave a  simpler proof of  Theorem~\ref{t:GGRS}. Both proofs rely on the following celebrated result,  obtained independently by Kostochka~\cite{Kostochka82,Kostochka84} and Thomason~\cite{Thomason84}. 

\begin{thm}[\cite{Kostochka82,Kostochka84,Thomason84}]\label{t:KT} For every integer  $t \geq 1$, every graph with no $K_t$ minor is $O(t\sqrt{\log{t}})$-degenerate. 
\end{thm}
   
Note that Theorem~\ref{t:KT} directly implies that every graph with no $K_t$ minor is $O(t\sqrt{\log{t}})$-colorable. 
Very recently, the present authors~\cite{NorSong19}  made  the first improvement on the order of magnitude of this  bound. Shortly after, Postle~\cite{Pos19} further improved a major part of the argument from~\cite{NorSong19} showing the following. 

\begin{thm}[\cite{Pos19}]\label{t:main} For every $\beta > \frac 1 4$ and every integer  $t \geq 1$,
every graph with no $K_t$ minor is $O(t (\log t)^{\beta})$-colorable.
\end{thm}

In this paper we combine the results and ideas of~\cite{GGRSV08, NorSong19, Pos19} to extend Theorem~\ref{t:main} to odd minors, as follows.

\begin{thm}\label{t:oddmain} For every $\beta > \frac 1 4 $ and every integer  $t \geq 1$,
	every graph with no odd $K_t$ minor is $O(t (\log  t)^{\beta})$-colorable.
\end{thm}

The proof of \cref{t:oddmain} occupies the rest of the paper. 

\subsubsection*{Notation}

We use largely standard graph-theoretical notation. We denote by $\v(G)$, $\e(G)$, $\chi(G)$, and  $\kappa(G)$  the number of  vertices,  number of  edges,   chromatic number, and (vertex) connectivity of a graph $G$, respectively.     We use $\d(G)=\e(G)/\v(G)$ to denote the \emph{density} of a non-null graph $G$, and $G[X]$ to denote the subgraph of a graph $G$ induced by a set $X \subseteq V(G)$. 
For a positive integer $n$, let $[n]$ denote $\{1,2,\ldots,n\}$. The logarithms in the paper are natural unless specified otherwise.

Let $H$ and $G$ be graphs.  An \emph{$H$-expansion} in  $G$ is a function $\eta$ with domain $V(H) \cup E(H)$ such that \begin{itemize}
 	\item for every vertex $v \in V(H)$, $\eta(v)$ is a subgraph of $G$ which is a tree, and the trees $\{\eta(v)\}_{v \in V(H)}$ are pairwise vertex-disjoint, and 
 	\item for every edge $e=uv \in E(H)$, $\eta(e)$ is an edge of $G$ with one end in $V(\eta(u))$ and the other in  $V(\eta(v))$.	
 \end{itemize}
We call the trees $\{\eta(v): v\in V(H)\}$ the \emph{nodes} of the expansion, and denote by $\cup\eta$ the subgraph of $G$ with vertex set $\bigcup_{v\in V(H)} V(\eta(v))$ and edge set $ \{E(\eta(v)): v\in V(H)\}\cup  \{\eta(e): e\in E(H)\}$. An $H$-expansion $\eta$ is \emph{bipartite} if $\cup\eta$  is bipartite. Moreover, we say that an $H$-expansion $\eta$ in $G$ is \emph{$S$-rooted} for $S \subseteq V(G)$  if $|S|=\v(H)$  and  $|V(\eta(v)) \cap S|=1$ for every $v \in V(H)$. 
 
It is well-known and easy to see that  $G$ has an $H$ minor if and only if there is an $H$-expansion in $G$.
We say that $G$ has a \emph{bipartite $H$ minor} if   there exists a bipartite $H$-expansion in $G$.

\section{Proof of~\cref{t:oddmain}}\label{s:main}
In this section we  prove \cref{t:oddmain}, pending the proof of   a key technical result, \cref{t:minorfrompieces}, which is  proved in \cref{s:build}.

We use the same strategy as in~\cite{NorSong19}, where we established a new    upper bound on the chromatic number of graphs with no $K_t$ minor, and we refer the reader to~\cite[Section 2]{NorSong19} for the outline of the argument. Several parts of the proof, however, become much more involved. In particular, in~\cite{NorSong19} we could easily reduce the proof to the case when the graph has high connectivity. Here we shall introduce a non-standard technical  notion of connectivity which we now present.

Recall that the pair $(A,B)$   is a \emph{separation} of a graph $G$  if  $A \cup B = V(G)$ and every edge of $G$ has both ends in $A$ or   both in $B$.  A separation  $(A,B)$ of $G$ is \emph{proper} if $A-B \neq  \emptyset$ and $B-A \neq \emptyset$. The \emph{order} of a  separation $(A,B)$ is $|A \cap B|$. We say that a graph $G$ is \emph{weakly $k$-connected} if for every proper separation $(A,B)$ of $G$ of order at most $k$, we have $$\min\{|A-B|,|B-A|\} < |A \cap B|.$$  

Our first lemma   ensures that every  graph $G$ with large chromatic number contains a subgraph $H$ with high weak connectivity such that the chromatic number of $H$ is still fairly large. All colorings in the proof of   \cref{l:weakconnect} are proper vertex-colorings.

\begin{lem}\label{l:weakconnect}
Let $k, l$ be positive integers with $k \geq 3l$. Then every graph  $G$ with $\chi(G)>k$ contains a  weakly $l$-connected subgraph $H$ such that  $$\chi(H)>k-2l.$$
\end{lem}

\begin{proof}  Let $H$ be a subgraph of $G$ with $\v(H)$ minimum such that for some $Z\subseteq V(H)$ with $|Z|\le 2 l$ there exists a  $k$-coloring $\phi_Z : Z \to [k]$ of  $H[Z]$ which cannot be extended to a  $k$-coloring of $H$.    It is easy to see that such a subgraph $H$ exists, and satisfies $\chi(H) > k-2l$. 
	
It remains to show that $H$ is weakly $l$-connected. Suppose $H$ is not weakly $l$-connected. Then there exists a proper separation $(A,B)$ of $H$ of order at most $l$ such that \[\min\{|A-B|,|B-A|\} \geq |A \cap B|.\]
   We may   assume that $|Z \cap (B-A)| \leq l$. 
Suppose first that $B-A \ne  Z \cap (B-A)$. 
By the choice of $H$,  there exists a $k$-coloring $\phi_A: Z \cup A \to [k]$ of $H[Z \cup A]$ such that $\phi_A$ extends $\phi_Z$. Let $Z' = (Z \cup A)  \cap B$. Then \[|Z'|=|Z\cap (B-A)|+|A\cap B| \leq 2l.\] By the choice of $H$, $\phi_A|_{Z'}$ can be extended to a $k$-coloring of $H[B]$. Thus  $\phi_A$ (and so $\phi_Z$) can be extended to a $k$-coloring of $H$, a contradiction.
This proves that $B-A = Z \cap (B-A)$. 
Let $Z'' =  (Z \cup B)  \cap A$.  Then 
\[|Z'' \cup Z| \le |Z|+|A\cap B|\leq 3l \leq k.\]
It follows that $A-(Z''\cup Z)\ne \emptyset$ because $\chi(H)>k$. By the choice of $H$, 
  there exists a $k$-coloring $\phi $ of $H[Z'' \cup Z]$ such that $\phi$ extends  $\phi_Z$. 
Note that $|B-A| \ge   |A \cap B|$ and so $|Z''| \le  |Z| - |B-A| +  |A \cap B| \leq |Z|\le 2l$. By the minimality of $H$,  the coloring $\phi|_{Z''}$ can be further extended to a $k$-coloring of $H[A]$. This  yields a $k$-coloring of $H$ extending  $\phi_Z$, a contradiction.  
\end{proof}

Our first application of \cref{l:weakconnect},  combined with the following result of Geelen et al. \cite{GGRSV08},  allows us to convert bipartite clique minors to odd clique minors.

\begin{thm}[\cite{GGRSV08}]\label{t:BipartiteToOdd}
	If   a graph  $H$  has a bipartite $K_{12t}$ minor, then either $H$ contains an odd $K_t$ minor, or there exists $X \subseteq V(H)$ with $|X| \le  8t-2$ such that some component of $H\setminus X$ is bipartite and contains at least $8t+2$ vertices.
\end{thm}

\begin{cor}\label{c:BipartiteToOdd} Let $k,t$ be positive integers with $k \geq 16t$. Assume that 
every graph $H$ with $\chi(H) > k$ has either a bipartite $K_{12t}$ minor or an odd $K_t$ minor. Then every graph $G$ with  $\chi(G) > k+16t$ has an odd $K_t$ minor.  
\end{cor}	

\begin{proof} Let $G$ be a graph with $\chi(G)> k + 16t$. Suppose for a contradiction that $G$ has no odd $K_t$ minor. By \cref{l:weakconnect},  $G$ contains a weakly  $8t$-connected subgraph $H$ with $\chi(H)>k$. Note that  $H$ has no odd $K_t$ minor.  By our assumption, $H$ has a bipartite $K_{12t}$ minor. Furthermore,
by \cref{t:BipartiteToOdd}, there exists a  proper separation $(A,B)$ of $H$ of order at most $8t-2$  such that $H[B-A]$ is bipartite, and $|B-A| \geq 8t+2 > |A \cap B|$. Since $H$ is  weakly $8t$-connected, we see that $|A-B| < |A \cap B|$. But then $$\chi(H) \leq \chi(H[B-A]) + \chi(H[A\cap B])+\chi(H[A-B]) < 16t-2 <k,$$ 
a contradiction.
\end{proof}	

The second ingredient of the proof of  \cref{t:oddmain} allows us to deal with the case when  $\v(G)$ is small. It is based on the following bound due to Kawarabayashi and the second author~\cite{KawSong07} on the independence number of graphs with no odd $K_t$ minor.

\begin{thm}[\cite{KawSong07}]\label{t:OddAlpha}
	Let $G$ be a graph with no an odd $K_t$ minor.  Then  $\alpha(G)  \geq  \v(G)/(2t)$. 
\end{thm}

\cref{t:OddAlpha} implies the following bound on the chromatic number of graphs with no odd $K_t$ minor.

\begin{cor}\label{c:big}
 	Let $G$ be a graph with no odd $K_{t}$ minor.  Then \begin{equation}\label{e:big}
 	\chi(G) \leq 2t \left(1+\log\left(\frac{\v(G)}{t}\right) \right).
 	\end{equation}
\end{cor}
 
\begin{proof}
 	By \cref{t:OddAlpha},  for every integer $s \geq 1$, there exist pairwise disjoint, independent subsets $X_1,X_2, \ldots,X_s \subseteq V(G)$ such that \[|V(G) - \cup_{i=1}^{s}X_i| \leq (1-1/(2t))^s\cdot \v(G).\]
 	Let $s=\lceil 2t \cdot \log(\v(G)/t) \rceil$. Then  $(1-1/(2t))^s\cdot \v(G) \leq t$.   It follows that  $\chi(G\setminus\cup_{i=1}^{s}X_i) \leq t$ and so  $$\chi(G) \leq  \chi(G[\cup_{i=1}^{s}X_i]) +\chi(G\setminus\cup_{i=1}^{s}X_i) \leq s+t,$$ as desired.	
\end{proof}

The third ingredient is~\cref{t:minorfrompieces} on the existence of  bipartite $K_{t}$ minors in weakly $l$-connected graphs.  The proof of~\cref{t:minorfrompieces} is more involved and will be given   in~\cref{s:build}.

\begin{thm}\label{t:minorfrompieces} There exists a constant $C=C_{\ref{t:minorfrompieces}}$ satisfying the following. 
	Let $t,l,r \geq 2$ be integers such that  $r \geq \sqrt{\log t}/2$ and $l = C t(\log t)^{1/4}$. Let $G$ be a weakly $l$-connected graph. If 
	there exist pairwise disjoint sets $X_1,X_2,\ldots,X_{r} \subseteq V(G)$ such that $\d(G[X_i]) \geq l$ for every   $i \in [r]$, and $\chi(G \setminus \cup_{i\in [r]} X_i) \geq l$,  then $G$ has a bipartite $K_t$ minor. 
\end{thm}

Our fourth tool is a bound from~\cite{GGRSV08} on the density sufficient to force a bipartite $K_{t}$ minor. Note  that  Kostochka~\cite{Kostochka82}  proved that  every graph $G$  with $\d(G) \geq 3.2 s \sqrt{\log s}$ has a $K_s$ minor.

\begin{thm}[\cite{GGRSV08}]\label{t:BipartiteFromDensity}  
	Every graph $G$ with $\d(G) \geq 7t \sqrt{\log t}$  contains a bipartite $K_{t}$ minor.
\end{thm}

Finally, we use the result of Postle~\cite{Pos19}, which is an improvement of a similar theorem from~\cite{NorSong19}.

\begin{thm}[\cite{Pos19}] \label{t:newforced} For every $\delta > 0$ there exists $C=C_{\ref{t:newforced}}(\delta) > 0$ such that for every $D > 0$ the following holds. Let $G$ be a graph with $\d(G) \ge  C$, and let $s=D/\d(G)$.  Then $G$ contains  either 
\begin{enumerate}
\item[(i)]   a minor $J$ with $\d(J) \geq D$, or
		\item[(ii)]    a subgraph $H$ with $\v(H) \leq s^{1+\delta} CD$ and $\d(H) \geq s^{-\delta}\d(G)/C$.
	\end{enumerate}  
\end{thm}

In the remainder of this section we deduce   \cref{t:oddmain}. It
 follows immediately from  \cref{c:BipartiteToOdd} 
 and the following \cref{t:main2}. The proof of \cref{t:main2}   utilizes all the tools presented above.  

 \begin{thm}\label{t:main2} For every $\delta>0$  there exists $t_0=t_0(\delta)$ such that for all positive integers $t \geq t_0$,  every graph $G$ with neither a bipartite $K_{t}$ minor nor an odd $K_t$ minor satisfies $$\chi(G)< t(\log t)^{\frac14   + \delta}.$$
  \end{thm}	
  
  \begin{proof} We may assume that $\delta < 1/4$.
  Let $C_1 = C_{\ref{t:newforced}}(\delta)$  and   $C_2=C_{\ref{t:minorfrompieces}}$. We choose $t_0 \gg \max\{C_1,C_2,1/\delta\}$ implicitly to satisfy the inequalities appearing throughout the proof.

  Let $t \geq t_0$ be an integer and  let $k = t(\log t)^{\frac14   + \delta}/6$.  Suppose for a contradiction that there exists  a graph $G$ with neither a bipartite $K_{t}$ minor nor an odd $K_t$ minor such that $\chi(G) \geq 6k$. By \cref{l:weakconnect},   $G$ contains a weakly $2k$-connected subgraph $H$ with $\chi(H) \geq 2k$. Our goal is to apply \cref{t:minorfrompieces} and \cref{t:newforced} to  $H$ to obtain a contradiction. 
   
  Choose a maximal collection $\{X_1,X_2,\ldots,X_{r'}\}$ of pairwise disjoint subsets of $V(H)$ such that $\d(H[X_i]) \geq C_2\cdot t  (\log t)^{1/4}$ and $|X_i| \leq t(\log t)^{3/4}$ for all $i\in[r']$. 
  Let $r= \lceil \sqrt{\log t}/2 \rceil$,   $r^*=\min\{r', r\}$,  and  $X = \cup_{i \in [r^*]}X_i$. Then $|X| \leq t (\log t)^{5/4}$.
  By \cref{c:big},  for sufficiently large $t$,   	\[\chi(H[X]) \leq 3t \cdot \log(\log t) \leq k-1.\] Thus $\chi(H \setminus X) \geq k+1$. 
If  $r'\ge r$, then  $\chi(H \setminus \cup_{i \in [r]}X_i) \geq k+1 \geq C_2\cdot t(\log t)^{1/4}$. By \cref{t:minorfrompieces} applied to $H$ and $\{X_1,X_2,\ldots,X_{r}\}$, we see that  $H$ has a bipartite $K_t$ minor, contrary  to the choice of $G$.
  Thus  $r'<r$. Then $X=  \cup_{i \in [r']}X_i$. \medskip
  
 Let $H'$ be a minimal subgraph of $H \setminus X$ with $\chi(H') \geq k+1$. Then the minimum degree of    $H'$  is  at least $k$, and so $\d(H') \geq k/2 = t(\log t)^{\frac14   + \delta}/12$.  
  Let $D= 7t \sqrt{\log t}$. We next apply \cref{t:newforced} to $
  D$ and $H'$. Note that  $H'$ has no  bipartite $K_{t}$ minor by the choice of $G$. Thus by \cref{t:BipartiteFromDensity},    $H'$ has no minor   $J$ with $\d(J) \geq D$.
  By \cref{t:newforced},  there must exist  $Z \subseteq V(H')$ such that  $|Z| \leq s^{1+\delta} C_1D$ and $\d(H[Z]) \geq s^{-\delta}\d(H')/C_1$, where $s=D/d(H') \leq 100(\log{t})^{1/4-\delta}$.  It is easy to check that,  for sufficiently large $t$, the above conditions yield that $\d(H[Z]) \geq  C_2\cdot t(\log t)^{1/4}$ and $|Z| \leq t(\log t)^{3/4}$. But then the collection  $\{X_1,X_2,\ldots,X_{r'},Z\}$ contradicts the maximality of $\{X_1,X_2,\ldots,X_{r'}\}$.
  \end{proof}

\section{Proof of Theorem~\ref{t:minorfrompieces}}\label{s:build}

The proof of \cref{t:minorfrompieces} is an adaptation of the proof of \cite[Lemma 3.3]{NorSong19}, which in turn is based on the ideas of Thomason~\cite{Thomason01}. In addition to some of the lemmas from the previous section we use an array of extra tools from the literature, which we now present. 

We first use a classical result of Mader,  which ensures  a highly-connected subgraph in a dense graph, to deduce a highly-connected bipartite subgraph in a graph with either high density or large chromatic number.

\begin{lem}[\cite{Mader72}]\label{l:connect}
	Every graph $G$ contains a subgraph $H$ such that $\kappa(H) \geq \d(G)/2$.
\end{lem}

\begin{cor}\label{c:connect} Every graph $G$ contains a bipartite subgraph $H$ such that 
		$$\kappa(H) \geq \max \{\d(G)/4,(\chi(G)-1)/8\}.$$
\end{cor}
\begin{proof}
	By a well-known result of Erd\H{o}s~\cite{Erdos65}, $G$ contains a bipartite subgraph $G'$ with $\d(G') \geq \d(G)/2$. By \cref{l:connect},  $G'$ contains a subgraph $H$ with $\kappa(H) \geq \d(G)/4$. 
	
	Next, let $k=\chi(G)$, and let $G''$ be a minimal subgraph of $G$ such that $\chi(G'')=k$. Then   the  minimum degree of $G''$ is at least $k-1$  and so $\d(G'') \geq (k-1)/2$. As shown in the previous paragraph $G''$ contains a bipartite subgraph $H$ such that 
	$\kappa(H) \geq (k-1)/8$, as desired.  
\end{proof}

A large part of the proof of \cref{t:minorfrompieces} involves linking together small bipartite clique-expansions, which we find  in each  $G[X_i] $   given in the theorem. We now present the terminology and tools needed to accomplish this. 

Let $l$  be a positive integer and let $\mc{S} =(\{s_i,t_i\})_{i \in [l]}$ be a sequence of pairwise disjoint pairs of vertices of a graph $G$, except possibly $s_i=t_i$ for some $i\in[l]$.    \emph{An $\mc{S}$-linkage $\mc{P}$} in $G$ is a sequence  $(P_i)_{i \in [l]}
$ of vertex-disjoint paths in $G$ such that $P_i$ has ends $s_i$ and $t_i$ for every $i \in [l]$. For an $\mc{S}$-linkage $\mc{P}$, let $I$ be the set of all $i \in [l]$ such that $P_i$ has an odd number of edges. Then we say that 
$\mc{P}$ is \emph{an $(\mc{S}, I)$-parity linkage}.  Let $S\subseteq V(G)$ with $|S|=2l$.   We say that  $S$ with $|S|=2l$ is \emph{linked} in   $G$ if for every ordered partition $\mc{S} =(\{s_i,t_i\})_{i \in [l]}$ of $S$ into pairs there exists an $\mc{S}$-linkage in $G$; and  $S$ is \emph{parity-linked} if, in addition,     for every $I \subseteq [l]$ there exists an $(\mc{S},I)$-parity linkage in $G$.   Finally,  a graph $G$ with $|V(G)| \geq 2l$ is \emph{$l$-linked} if every set $S \subseteq V(G)$ with $|S|=2l$ is linked in $G$.

Our second tool is the following theorem of Thomas and Wollan~\cite{ThoWol05}, which improves an earlier result of Bollob\'as and Thomason~\cite{BolTho96}.

\begin{thm}[\cite{ThoWol05}]\label{t:linked}
	For every integer $l \geq 1$,  every graph $G$ with $\kappa(G) \geq 10l$ is $l$-linked.
\end{thm}

 Kawarabayashi and Reed~\cite{KawReed09} extended~\cref{t:linked} to parity linkages. They proved that for  every graph $G$ with $\kappa(G) \geq 50l$,  either there exists $X \subseteq V(G)$ such that $|X| < 4l-3$ and $G \setminus X$ is bipartite, or every set $S \subseteq V(G)$ with $|S|=2l$ is parity-linked in $G$. We need a variant of their result for weakly connected graphs. Fortunately, we are able to reuse most of the ingredients of the proof from~\cite{KawReed09}. One of these ingredients is the ``Erd\H{o}s-Pos\'a property" for odd $S$-paths.

\begin{thm}[\cite{CGGGLS06,GGRSV08}]\label{t:parity1} Let $k \geq 1$ be an integer. For any   set $S $ of vertices of a graph $G$, either 
	\begin{enumerate} 
		\item[(i)] there are  $k$ vertex-disjoint paths  each of which has 
	an odd number of edges and both its ends in $S$, or
	\item[(ii)] there exists $X \subseteq V(G)$ with  $|X| \leq 2k-2$ such that $G \setminus X$ contains
	no such path.
	\end{enumerate}
\end{thm}

Let $G$ be a graph, and let $H$ be a bipartite subgraph of $G$. We say that a path $P$   in $G$ is a \emph{parity-breaking path for $H$}, if the ends of $P$ are in $V(H)$, $P$ is otherwise vertex-disjoint from $H$, and $H \cup P$ contains an odd cycle. Note that such a parity-breaking path may have only one edge.  
For a graph $G$ and  $X,Y \subseteq V(G)$,  we say that $X$ is \emph{joined to $Y$ in $G$} if there exist $|X|$ vertex-disjoint paths in $G$ each of which has   one end in $X$ and the other in $Y$.    We need the following lemma  which  is a consequence of the result of Kawarabayshi and Reed~\cite[Theorem 1.2]{KawReed09}.
 
\begin{lem}[\cite{KawReed09}]\label{l:KawReed}  Let $H$ be a $2k$-linked bipartite subgraph of a graph $G$. Suppose that $G$ contains $2k$ vertex-disjoint parity-breaking paths for $H$. Then every set $X \subseteq V(G)$ with $|X| = 2k$ that is joined to $V(H)$ in $G$ is parity-linked in $G$.
\end{lem}

We say that a set $X $ of vertices of a graph $G$  is \emph{parity-knitted} if for every pair of partitions $(A,B)$ and $(X_1,X_2,\ldots,X_r)$  of $X$,  there exist   pairwise vertex-disjoint trees $T_1, T_2,\ldots,T_r$ in $G$  such that $X_i \subseteq V(T_i)$  and $(A \cap X_i,B \cap X_i)$ extends to the bipartition of $T_i$    for every $i \in [r]$.  A \emph{linkage} in a graph $G$ is a collection of pairwise vertex-disjoint paths.

\begin{cor}\label{c:parityknitted}
	Let $H$ be a $2k$-linked bipartite subgraph of a graph $G$. Suppose that $G$ contains $2k$ vertex-disjoint parity-breaking paths for $H$. Then every set $X \subseteq V(G)$ with $|X| = k$ that is joined to $V(H)$ in $G$ is parity-knitted in $G$.  
\end{cor}

\begin{proof} Let $(A,B)$ and $(X_1,X_2,\ldots,X_r)$ be two partitions of $X$. 
	As $X$ is joined to $V(H)$ in $G$,  there exists  a linkage  $\mc{P}$ in $G$ such that $|\mc{P}|=k$, and every path in $  \mc{P}$ has one end in $X$, the other   in $V(H)$, and is otherwise vertex-disjoint from $ H$. 
	Let $Y=\{V(P)\cap V(H)| P\in\mc{P}\}$.  Then $|Y|=k$.    Note that the minimum degree of $H$ is  at least  $4k-1$ because $H$ is $2k$-linked. It follows that  we can greedily find pairwise vertex-disjoint trees $T'_1,T'_2,\ldots T'_r$ in $H \setminus Y$ such that $\v(T'_i)=|X_i|$ for each $i \in [r]$. Let $X' = X \cup (\cup_{i \in r} V(T'_i))$. Then $|X'|=2|X|=2k$ and $X'$ is joined to $V(H)$ in $G$.  By \cref{l:KawReed}, $X'$ is parity-linked in $G$.
Thus  we can find pairwise  vertex-disjoint linkages $\mc{P}_1,\mc{P}_2,\ldots,\mc{P}_r$ in $G$ such that for each $i\in[r]$, $|\mc{P}_i|=|X_i|$;    every  path in $\mc{P}_i$ has one end in $X_i$ and    the other   in $V(T'_i)$,  and is otherwise  disjoint from $X'$;   and in addition,  by choosing  the desired parity of each path in $\mc{P}_i$, the partition  $(A \cap X_i,B \cap X_i)$ extends to the bipartition of the tree $T_i =  T'_i \cup (\cup_{P \in \mc{P}_i}P)$, as desired. 	
\end{proof}

Finally, we need a lemma from~\cite{NorSong19}.
 
\begin{lem}[\cite{NorSong19}]\label{l:rooted} There exists a constant $C=C_{\ref{l:rooted}} >0$ satisfying the following. 
	Let $G$ be a graph,  let $m, s \geq 2$ be positive integers.  Let $s_1,\ldots, s_m,$ $t_1,\ldots,t_m,$  $r_1,\ldots, r_s \in V(G)$ be distinct, except possibly $s_i=t_i$ for some $i \in [m]$.
	If $\kappa(G) \geq C\cdot  \max\{m, s\sqrt{\log s}\}$,  	
	then there exists a $K_{s}$-expansion $\eta$ in $G$ rooted at $\{r_1,\ldots,r_s\}$ and an $(\{s_i,t_i\})_{i \in [m]}$-linkage $\mc{P}$ in $G$ such that $\cup\eta$ and $\mc{P}$ are vertex-disjoint. 
\end{lem}

We are  now ready to prove \cref{t:minorfrompieces} by building a bipartite $K_t$ minor   from the pieces constructed in each $G[X_i]$.  

\begin{proof}[Proof of \cref{t:minorfrompieces}] 
	We  show that the theorem holds for $C_{\ref{t:minorfrompieces}}=\max\{2000,48C_{\ref{l:rooted}}\}$.

Let $k = t(\log t)^{1/4}$. Then $l \geq \max\{2000,48C_{\ref{l:rooted}}\}\cdot k$.
By \cref{l:weakconnect},   $G \setminus \cup_{i\in [r]} X_i$ contains a  weakly $(l/3)$-connected subgraph $G_0$  with    $\chi(G_0) \geq l/3$. By  \cref{c:connect} and the choice of $C$,  $G_0$ contains a bipartite subgraph $H_0$ with $\kappa(H_0) \geq (l-3)/24>80k$.   We next show  that \medskip

  \noindent ($*$) $G_0$ contains at least $8k$ vertex-disjoint parity-breaking paths for $H_0$.
 \medskip

Suppose ($*$) is not true. Let $(A_0,B_0)$ be a bipartition of $H_0$. Then $ |A_0|\ge \kappa(H_0) > 80k$.  Note that every path in $G_0$ with an odd number of edges and both its ends in $A_0$ contains   a parity-breaking path for $H_0$. Thus $G_0$ does not contain $8k$   vertex-disjoint paths each of which has an odd number of edges and both its ends in $A_0$.  By \cref{t:parity1} applied to $G_0$ and $A_0$,  there exists $X \subseteq V(G_0)$ with $|X| \le 16k-2$ such that $G_0 \setminus X$ contains no such path.  As  $\kappa(H_0) > 80k$ and $\chi(G_0 \setminus X) \geq \chi(G_0) - |X| \geq l/4$,  it follows that   the block of $G_0 \setminus X$ containing $H_0 \setminus X$ is bipartite, and $G_0 \setminus X$ contains a block $J$ with  $\chi(J) \geq l/4$.   Now consider a proper separation $(A_1,B_1)$ of $G_0$ with  $V(H_0) \subseteq A_1$ and  $V(J) \subseteq B_1$  such that  $|A_1\cap B_1|\le |X|+1 < 16k<l/3$. Then $|A_1-B_1|\ge \v(H_0) -|A_1\cap B_1|  > 16k$.   Since $G_0$ is  weakly $(l/3)$-connected, we have  $|B_1-A_1|< |A_1\cap B_1|$ and so  $\v(J) \leq |B_1| = |A_1\cap B_1|+|B_1-A_1| <32k$.   But then  $\chi(J) < 32k < l/4$, a contradiction. This proves  ($*$).  \medskip

Let $y = \lfloor (\log t)^{1/4} \rfloor$ and  $x  = \lceil t/ y \rceil$. Assume first that  $y \leq 1$. Then $G[X_1]$ contains a bipartite $K_t$ minor by \cref{t:BipartiteFromDensity}, as desired. Assume next that $y \geq 2$.  Then $r \geq \binom{y}{2}$, $xy \geq t$, and $xy(y-1) \leq 4k$.  It suffices to show that $G$ contains a bipartite $K_{xy}$-expansion. 

 By  \cref{c:connect},  $G[X_i]$ contains a bipartite subgraph $H_i$    with $\kappa(H_i) \geq l/4$ for each $i \in [r]$.      Let $\mc{H}=\{H_1,H_2, \ldots,H_{\binom{y}{2}}\}$. We relabel the graphs in $\mc{H}$ to $\{H_{\{i,j\}}\}_{\{i,j\} \subseteq [y]}$.   
We   claim that there exist pairwise vertex-disjoint linkages $\mc{Q}_{(i,j)}$ for all $i,j \in [y]$ with $i \neq j$, such that each $\mc{Q}_{(i,j)}$ consists of $x$ vertex-disjoint paths $Q^{1}_{(i,j)}, \ldots, Q^{x}_{(i,j)} $ each of which starts in $V(H_{\{i,j\}})$, ends in $V(H_0)$, and is otherwise vertex-disjoint from $ H_{\{i,j\}} $. Suppose not. By Menger's theorem  there exists a   proper  separation $(A,B)$ of $G$  with $|A\cap B|<xy(y-1) \leq 4k<l$ such that $V(H_0) \subseteq A$ and $V(H) \subseteq B$ for some $H \in \mc{H}$. But then \[\min\{|A-B|, |B-A|\}>|A\cap B|,\] contrary to  the fact that   $G$ is weakly $l$-connected.

Let $\mc{Q} = \cup_{i,j \in [y], i \neq j} \mc{Q}_{(i,j)}$.   We now apply \cref{l:rooted} consecutively to each   subgraph $H_{\{i,j\}}$ with $s=2x$  and $m \leq xy(y-1)-2x$ equal to the number of paths in $\mc{Q} - (\mc{Q}_{(i,j)} \cup \mc{Q}_{(j,i)})$   which are not vertex-disjoint from $H_{\{i,j\}}$.  The vertices $\{(s_i,t_i)\}_{i \in [m]}$ are then chosen to be the first and last vertices of these paths in $H_{\{i,j\}}$, while the vertices $r^1_{(i,j)}, \ldots,r^x_{(i,j)}$ and   $r^1_{(j,i)}, \ldots,r^x_{(j,i)}$ are the ends of the paths in  $\mc{Q}_{(i,j)}$  and $ \mc{Q}_{(j,i)} $ in $H_{\{i,j\}}$, respectively. Note that  $\kappa(H_{\{i,j\}}) \geq l/4 >C_{\ref{l:rooted}}\cdot  \max\{m, s\sqrt{\log s}\} $  by the choice of $C$.  By using  the linkage $\mc{P}$ given by  \cref{l:rooted} to reroute the paths in $\mc{Q} - (\mc{Q}_{(i,j)} \cup \mc{Q}_{(j,i)})$ within $H_{\{i,j\}}$, we may assume that   $H_{\{i,j\}}$ contains a   bipartite  $K_{2x}$-expansion $\eta_{_{\{i,j\}}}$  rooted at $\{r^1_{(i,j)}, \ldots,r^x_{(i,j)}, r^1_{(j,i)}, \ldots,r^x_{(j,i)}\}$, which is  otherwise vertex-disjoint from $\mc{Q}$.  We may assume that $\eta_{_{\{i,j\}}}$ has domain $\{i,j\} \times [x]$ such that   for each $z \in [x]$, we have $r^{z}_{(i,j)} \in V(\eta_{_{\{i,j\}}}(i,z))$ and  $r^{z}_{(j,i)} \in V(\eta_{_{\{i,j\}}}(j,z))$.  

For $i,j \in [y]$ with  $i \neq j$ and $z \in [x]$,  let $s^{z}_{(i,j)}  $ be  the first vertex of $H_0$ encountered as we traverse $Q^{z}_{(i,j)}$ from the end $r^{z}_{(i,j)}$ in $H_{\{i,j\}}$; let $R^{z}_{(i,j)}$ be the subpath of $Q^{z}_{(i,j)}$ with ends $r^{z}_{(i,j)}$ and $s^{z}_{(i,j)}$; and let $S_i^z=\{s^{z}_{(i,j)}:  j \in [y]-\{i\}\}$.   Finally, let
\[ S=\bigcup_{i  \in [y],  z\in[x] }  S_i^z  \text{ and   } 
  H^{*}= \bigcup_{\{i, j\} \subseteq [y] }\left((\cup \eta_{_{\{i,j\}}})\cup \{R^{z}_{(i,j)}\cup R^{z}_{(j,i)}: z\in[x] \} \right).\] 
It is easy to see that   $(S_i^z)_{i \in [y], z \in [x]}$  partitions $S$, and   $H^*$ is  a bipartite subgraph of $G $  with  $V(H^*) \cap V(H_0) = S$. Let $(A^*, B^*)$ be a bipartition of $H^*$. 
Note that $|S|=  xy(y-1)\leq 4k$ and $S$ is joined to $V(H_0)$ in $G_0$. By \cref{t:linked},    $H_0$ is   $8k$-linked because $  \kappa(H_0)>80k$.  Furthermore, by ($*$) and \cref{c:parityknitted}, $S$ is parity-knitted in $G_0$. Thus for the pair of partitions $(A^*\cap S, B^*\cap S)$ and $(S_i^z)_{i \in [y], z \in [x]}$ of $S$, there exists a collection  $\{T^z_i\}_{i \in [y], z \in [x]}$ of pairwise vertex-disjoint trees  in $G_0$  such that $S_i^z  \subseteq V(T^z_i)$,   and   $G^* =H^*\cup (\cup_{i \in [y],z \in [x]}T^z_i) $   is  a bipartite subgraph of $G$. 
 
It remains to show that $G^*$ contains a $K_{xy}$-expansion. We construct such an expansion $\eta$ with domain $[y] \times [x]$. 
  For $i \in [y]$ and $ z \in [x]$, let   
	 $$\eta(i,z)= T^z_i \cup \bigcup_{j \in [y]-\{i\}}\left(\eta_{_{\{i,j\}}}(i,z) \cup R^{z}_{(i,j)}\right).$$
It is not hard to see that $\{\eta(i,z)\}_{i \in [y], z \in [x]}$ is a collection of pairwise vertex-disjoint trees in $G^*$.  Moreover, for each pair of distinct  elements $(i,z)$ and $(i',z')$ in  the domain of $\eta$, $G^*$ contains an edge with one end in $V(\eta(i,z))$ and the other  in $V(\eta(i',z'))$. Indeed, if $i=i'$, then $z\ne z'$ and   for each  $j \in [y]-\{i\}$, the edge in $\eta_{_{\{i,j\}}}$ with one end in $V(\eta_{_{\{i, j\}}}(i,z))$ and the other  in $V(\eta_{_{\{i, j\}}}(i,z'))$ is such  an edge; and, if $i \neq i'$,  then the edge in $\eta_{_{\{i,i'\}}}$ with one end in $V(\eta_{_{\{i, i'\}}}(i,z))$ and the other  in $V(\eta_{_{\{i, i'\}}}(i',z'))$ is the  desired one.   \medskip

This completes the proof of \cref{t:minorfrompieces}. \end{proof}

\subsubsection*{Acknowledgement.} The research presented in this paper was partially conducted during the visit of the second author to the Institute of Basic Science in Daejeon,  Korea.     The second author  thanks    the  Institute of Basic Science  for  the  hospitality, and   Sang-il Oum for helpful discussion.

\bibliographystyle{alpha}
\bibliography{Zixia}

\end{document}